\documentclass[a4paper,11pt]{article}
  \usepackage{amssymb,amsbsy,amsmath,latexsym,theorem,eepic,enumerate,times}
  \usepackage[all]{xy}
  \xyoption{curve}
  \usepackage{geometry}
  \usepackage[ps2pdf, bookmarks=true,
            bookmarksnumbered=true, breaklinks=true,
            pdfstartview=FitH, hyperfigures=false,
            plainpages=false, naturalnames=true,
            colorlinks=true,pagebackref=true,
            pdfpagelabels]{hyperref}
\usepackage[usenames]{color}
\def\mathscr{\mathcal}

{\theorembodyfont{\rmfamily}\newtheorem{example}{Example}[section]}
{\theorembodyfont{\rmfamily}\newtheorem{Def}[example]{Definition}}
\newtheorem{prop}[example]{Proposition}
\newtheorem{thm}[example]{Theorem}
{\theorembodyfont{\rmfamily}\newtheorem{rem}[example]{Remark}}
\newtheorem{cor}[example]{Corollary}
\newtheorem{lem}[example]{Lemma}

  \def\FTop{\mathsf{FTop}}

  \def\Crs{\mathsf{Crs}}
  
  \def\CRS{\mathsf{CRS}}
\def\Set{\mathsf{Set}}

\def\C{\mathsf{C}}

\def\geq{\geqslant}

\newenvironment{proof}{\noindent {\bf Proof }}{\rule{0mm}{1mm}\hfill $\Box$

\mbox{}}

\newcommand{\threeaxes}[3]{\def\objectstyle{\scriptstyle}  \objectmargin={0pt}
\xy
(0,0)*+{}="a",(0,-6)*+{\rule{0em}{1.5ex}#2}="b",(7,0)*+{\;#1}="c",
(14,-3)*+{\;#3}="d" \ar@{->} "a";"b" \ar @{->}"a";"c"  \ar
@{->}"a";"d"\endxy }

\newcommand{\directs}[2]{\def\objectstyle{\scriptstyle}  \objectmargin={0pt}
\xy
(0,4)*+{}="a",(0,-2)*+{\rule{0em}{1.5ex}#2}="b",(7,4)*+{\;#1}="c"
\ar@{->} "a";"b" \ar @{->}"a";"c" \endxy }

\newcommand{\xdirects}[2]{\def\objectstyle{\scriptstyle}  \objectmargin={0pt}
\xy
(0,0)*+{}="a",(0,-6)*+{\rule{0em}{1.5ex}#2}="b",(7,0)*+{\;#1}="c"
\ar@{->} "a";"b" \ar @{->}"a";"c" \endxy }
\newcommand{\sdirects}[2]{\def\objectstyle{\scriptstyle}  \objectmargin={0pt}
\xy
(0,2.2)*+{}="a",(0,-2.5)*+{\rule{0em}{1.5ex}#2}="b",(7,2.2)*+{\;#1}="c"
\ar@{->} "a";"b" \ar @{->}"a";"c" \endxy }

\def\hat{\widehat}

\def\I{\mathsf{I}}
\def\eps{\varepsilon}
\def\epsilon{\varepsilon}

\def\ge{\geqslant}

\def\hat{\widehat}

\def\eps{\varepsilon}
\def\epsilon{\varepsilon}

\def\tilde{\widetilde}

\def\ge{\geqslant}

\def\geq{\geqslant}

\def\wt{\widetilde}

\def\xybiglabels{\def\labelstyle{\textstyle}}

 \def\eps{\varepsilon}

\def\Gpds{\mathsf{Gpd}}

\def\bI{\mathbb{I}}

\def\bZ{\mathbb{Z}}

\def\ogpd{$\omega$-$\mathsf{Gpd}$}

\def\ogpdm{\omega\mbox{-}\mathsf{Gpd}}

\def\CAT{\mathsf{CAT}}

\def\epsilon{\varepsilon}

\def\sI{\mathcal{I}}
\def\Ab{\mathsf{Ab}}

\def\Cost{\operatorname{Cost}}

\begin{document}

\title{Covering morphisms  of  crossed
complexes and of cubical omega-groupoids are closed under   tensor
product}
\author{Ronald Brown\thanks{ School of
Computer Science, Bangor University, Gwynedd, LL57 1UT, UK; email:
r.brown@bangor.ac.uk . This  author would like to thank the
 Macquarie University for support of a visit in March, 1996, and
the   Leverhulme Foundation for support by an Emeritus Fellowship
2002-2004.} \and Ross Street\thanks{Department of Mathematics,
Macquarie University, NSW 2109, Australia; email:
ross.street@mq.edu.au. This author is grateful for partial support
of ARC Discovery Grant DP0771252. }}

\maketitle

\begin{abstract}The aim is the proof of the theorems of the title and the corollary
that the tensor product of two free crossed resolutions of groups or
groupoids is also a free crossed resolution of the product group or
groupoid. The route to this corollary  is through the equivalence of
the category of crossed complexes with that of cubical
$\omega$-groupoids with connections where the initial definition of
the tensor product lies. It is also in the latter category that we
are able to  apply techniques of dense subcategories  to identify
the tensor product of covering morphisms as a covering morphism.
\end{abstract}

\section*{Introduction}
A series of papers by R. Brown  and P.J. Higgins, surveyed in
\cite{Brown-grenoble,Brown-handbook}, has shown how the category
$\Crs$ of crossed complexes is a useful tool for certain nonabelian
higher dimensional local-to-global problems in algebraic topology,
for example  the calculation of homotopy 2-types of unions of
spaces; and also that crossed complexes are suitable coefficients
for nonabelian cohomology, generalising an earlier use of crossed
modules as coefficients. While crossed complexes have a long history
in algebraic topology, particularly in the reduced case, i.e. when
$C_0$ is a singleton, the extended use in these papers made them a
tool whose properties could be developed independently of classical
tools in algebraic topology such as simplicial approximation. A key
new tool for this approach was cubical, using the notion of cubical
$\omega$--groupoids with connections. A book is in press  on these
topics, \cite{BHS}.

One aspect of this work is that it leads to specific calculations of
homotopical and group theoretical invariants; as an example,  the
notion of identities among relations for a presentation of groups
combines both of these fields, since it also concerns the second
homotopy group $\pi_2(K(\mathcal P))$ of the 2-complex determined by
a presentation $\mathcal P$ of a group. Calculations of this module
were obtained in \cite{BRa99} not through `killing homotopy groups',
or its homological equivalent, finding generators of a kernel, but
through the notion of `constructing a home for a contracting
homotopy'. To this end we had to work by constructing a free crossed
resolution $\widetilde{F}$ of  the universal covering crossed
complex  of a group or groupoid. Any construction of a contracting
homotopy of $\widetilde{F}$ breaks the symmetry of the situation, as
is necessary, and also may rely on rewriting methods, such as
determining a maximal tree in the Cayley graph. Thus we see covering
crossed complexes as a basic tool in the application of crossed
complex methods, in analogy to the application of covering spaces in
algebraic topology.

A major tool for dealing with homotopies  is the construction of a
monoidal closed structure on the category $\Crs$ of crossed
complexes giving an exponential law of the form
\begin{equation*}
  \Crs(A \otimes B, C) \cong \Crs(A, \CRS(B,C))
\end{equation*}
for all crossed complexes $A,B,C$, \cite{BH87}.

This monoidal closed structure and the notion of classifying space
$BC$ of a crossed complex $C$  is applied in \cite{BH91} to give the
homotopy classification result
\begin{equation*} [X,BC] \cong [\Pi X_*, C]
\end{equation*}
where on the left hand side  with $X$  a CW-complex,  we have
topology, and  on the right hand side, with $\Pi X_*$ the
fundamental crossed complex of the skeletal filtration of $X$, we
have  the algebra of crossed complexes.

Tonks proved in \cite[Theorem 3.1.5]{T94} that the tensor product of
free crossed resolutions of a group is a free crossed resolution:
his proof used the  crossed complex Eilenberg-Zilber Theorem,
\cite[Theorem 2.3.1]{T94}, which was  published in \cite{tonks:EZ}.
The result on resolutions is applied in for example \cite{BP} to
construct some small free crossed resolutions of a product of
groups. We give here an alternative approach to this result.

The PhD thesis \cite{Day-thesis} of Brian Day addressed the problem
of extending a promonoidal structure on a category $\mathcal{A}$
along a dense functor $J:\mathcal{A}\rightarrow \mathcal{X}$ into a
suitably complete category $\mathcal{X}$ to obtain a closed monoidal
structure on $\mathcal{X}$. The two published papers
\cite{Day-midwest,Day-reflect}  are only part of the thesis and
represent components towards the density result. The formulas in,
and the spirit of, Day's work suggested our approach to the present
paper. However, here the category $\mathcal{A}$ is actually small
(consisting of cubes) and monoidal, and so is an easy case of Day's
general setting. The same simplification occurs in the approach to
the Gray tensor product of 2-categories in \cite{street-gray-2cat},
and of globular $\infty $-categories in \cite[Proposition
4.1]{crans-thesis}.

One  advantage of cubical methods is the standard formula
\begin{equation}\label{eq:tenscubes}
I^m_* \otimes I^n_* \cong I^{m+n}_*
\end{equation} where $I^m_*$ is the standard
topological $m$-cube with its standard skeletal filtration. This
equation is modelled in the category \ogpd\ by the formula
\begin{equation}\label{eq:tensomega cubes} \bI^m \otimes \bI^n \cong
\bI^{m+n}
\end{equation}
where for $m \geq 0$ $\bI^m$ is the free \ogpd\ on one generator
$c_m$ of dimension $m$. We  apply \eqref{eq:tensomega cubes} by
proving in Theorem \ref{thm:dense} that the full subcategory of
\ogpd\ on these objects $\bI^m, m \geq 0$, is dense in \ogpd . The
proof requires a further property of $\omega$-groupoids, that they
are $T$-complexes \cite{BH81:algcub,BH81:T}. We then use the methods
of Brian Day \cite{Day-reflect} to characterise the tensor product
on \ogpd\ as determined by the formula \eqref{eq:tensomega cubes}.

We use freely the notions and properties of ends and coends, for
which see \cite{Mac71}.

The final ingredient we need is the fact that if $p\colon  \wt{C}
\to C$ is a covering morphism of crossed complexes then $p^*\colon
\Crs/C \to \Crs/\wt{C}$ preserves colimits, since it has a right
adjoint. This result is due to Howie \cite{How79}, in fact for the
case of a fibration rather than just a covering morphism. Because of
the equivalence of categories, this applies also to the case of the
category \ogpd. However we need to characterise fibrations and
coverings in the category \ogpd. This is done in Section
\ref{sec:fibrandcov}. It is possible that the covering morphisms are
part of a factorization system as are the discrete fibrations in the
contexts of \cite{bourn-interngpds} and
\cite{street-verity-factorization}.
\section{Crossed complexes}
For the purposes of algebraic topology the most important feature of
the category $\Crs$ of crossed complexes  is the {\it fundamental
crossed complex functor}, \cite{BH81:col},
$$\Pi \colon  \FTop \to \Crs$$
from the category of filtered spaces
$$X_*\colon  X_0 \subseteq X_1 \subseteq \cdots \subseteq  X_n \subseteq \cdots\subseteq X_\infty .  $$
An extra assumption is commonly made that $X_\infty$ is the union of
all the $X_n$, but we do not use that condition. For such a filtered
space $X_*$, various relative homotopy groups
$$(\Pi X_*)_n(x)= \pi_n(X_n, X_{n-1},x)$$ for $x \in X_0$ and  $n
\geq 2$, may be combined with the fundamental groupoid $(\Pi X_*)_1=
\pi_1(X_1,X_0)$ on the set $X_0$ to give a crossed complex $\Pi
X_*$.  There are boundary operations $\delta_n\colon (\Pi X_*)_n \to
(\Pi X_*)_{n-1}$ and operations of $(\Pi X_*)_1 $ on $(\Pi X_*)_n, n
\geq 2$, satisfying axioms which are characteristic for crossed
complexes. This last fact follows because for every crossed complex
$C$ there is a filtered space $X_*$ such that $C \cong \Pi X_*$
\cite[Corollary 9.3]{BH81:col}.

The use of crossed complexes in the single vertex case continues
work of J.H.C. Whitehead,  \cite{W49:CHII,W50:SHT},  and of J.
Huebschmann,  \cite{Hu80}.

\section{Fibrations and covering morphisms of crossed complexes}
\label{secII4:FibCrs} The definition of fibration of crossed
complexes we are using is due to Howie in \cite{How79}; it requires
the definition of fibration of groupoids  given in \cite{B70,B2006},
generalising the definition of covering morphism of groupoids given
in \cite{H71}. The notion of fibration of crossed complexes given in
this Section leads to a Quillen model structure on the category
$\Crs$, as shown by Brown and Golasi{\'n}ski in \cite{BG89}, and
compared with model structures  on related categories in
\cite{aras-metayer-model}.

First recall that for a groupoid $G$ and object $x$ of $G$ we write
$\Cost_Gx$ for the union of the $G(u,x)$ for all objects $u$ of $G$.
A morphism of groupoids $p\colon H \to G$ is called a {\it fibration
(covering morphism)}, \cite{B70},  if the induced map  $\Cost_Hy \to
\Cost_Gpy$ is a surjection (bijection) for all objects $y$ of $H$.
(Here we use the conventions of \cite{BHS} rather than
 of \cite{B2006}.)

\begin{Def} 
A morphism $p \colon  D \to C$ of crossed complexes is a {\it
fibration (covering morphism)} if \begin{enumerate}[(i)] \item the
morphism $p_1 \colon  D_1 \to C_1 $ is a {\it fibration (covering
morphism) } of groupoids;
\item for each $n \geqslant 2$ and $y \in D_0$,
the morphism of
 groups $p_n \colon  D_n (y) \to C_n (py)$ is surjective (bijective).
\end{enumerate}
The morphism $p$ is a {\it trivial fibration} if it is a fibration,
and also a weak equivalence, by which is meant that $p$ induces a
bijection on $\pi _0$ and isomorphisms $\pi _1(D,y) \to \pi
_1(C,py)$, $H_n (D,y) \to H_n (C,py)$ for all $y \in D_0$ and $n
\geqslant 2$. \hfill $\Box$ \end{Def}

\begin{rem}  It is worth remarking that the
notion of covering morphism of groupoids appears in the paper
\cite[(7.1)]{PASmith1} under the name `regular morphism'.  Strong
applications of covering morphisms to combinatorial group theory are
given in \cite{H71}, and a full exposition  is also given in
\cite[Chapter 10]{B2006}.

A fibration of groupoids gives rise to a family of exact sequences,
\cite{B70,B2006}, which are extended  in \cite{How79} to a family of
exact sequences arising from a fibration of crossed complexes. These
latter exact sequences have been applied to the classification of
nonabelian extensions of groups in \cite{BMu}, and to the homotopy
classification of maps of spaces in \cite{brown-homclass}.
  \hfill $\Box$
\end{rem}
In Section \ref{sec:fibrandcov} we will need the following result,
which is an  analogue for crossed complexes of known results for
groupoids \cite[10.3.3]{B2006} and for spaces.
\begin{prop}\label{prop:II4liftmorphcov}
Let $p\colon  \wt{C} \to C$ be a covering morphism of crossed
complexes, and let $y \in\wt{C}_0$. Let $F$ be a connected crossed
complex, let $x \in F_0$, and let $f\colon  F \to C$ be a morphism
of crossed complexes such that $f(x)=p(y)$. Then the following are
equivalent:
\begin{enumerate}[\rm (i)]
\item $f$ lifts to a morphism $\tilde{f}\colon  F \to \wt{C}$ such that
$\tilde{f}(x)=y$ and $p\tilde{f}=f$;
\item $f(F_1(x)) \subseteq p(\wt{C}_1(y))$;
\item $f_*(\pi_1(F,x)) \subseteq p_*(\pi_1(\wt{C},y))$.
\end{enumerate}
Further, if the lifted morphism as above exists, then it is unique.
\end{prop}
\begin{proof} That (i)$\Rightarrow$ (ii) $\Rightarrow$ (iii) is
clear.

So we assume  (iii) and prove (i).

We first assume $F_0$ consists only of $x$. Then the value of
$\tilde{f}$ on $x$ is by assumption defined to be $y$.

Next let $a \in F_1(x)$. By the assumption (iii) there is $c \in
C_2(py)$ and $b \in \wt{C}(y)$ such that $f(a)= p(b) + \delta_2(c)$.
Since $p$ is a covering morphism there is a unique $d \in
\wt{C}_2(y)$ such that $p(d)=c$. Thus $f(a)= p(b+ \delta_2(d))$. So
we define $\tilde{f}(a)=b+ \delta_2(d) \in \wt{C}_2(y)$. It is easy
to prove from the definition of covering morphism of groupoids that
this makes $\tilde{f}$ a morphism $F_1(x) \to \wt{C}_1(y)$ such that
$p\tilde{f}=f$.

For $n \geq 2$ we define $\tilde{f}\colon  F_n(x) \to \wt{C}_n(y)$
to be the composition of $f$ in dimension $n$ and the inverse of the
bijection $p\colon  \wt{C}_n(y)\to C_n(py)$.

It is now straightforward to check that this defines a morphism
$\tilde{f}\colon  F,x \to \wt{C},y$ of crossed complexes as
required.

If $F_0$ has more than one point, then we choose for each $u$ in
$F_0$ an element $\tau_u \in F_1(u,x)$ with $\tau_x=1_x$. Then
$f(\tau_u)$   lifts uniquely to $\bar{\tau}_u\in \Cost_{
\widetilde{C}}y\,$: any lift $\tilde{f}\colon F,x \to
\widetilde{C},y$ of $f$ must satisfy $\tilde{f}(\tau_u)=
\bar{\tau}_u$ so we take this as a definition of $\tilde{f}$ on
these elements.

If $a \in F_1(u,v)$ then $a=\tau_u + a' -\tau_v$ where $a' \in
F_1(x)$ and so we define $\tilde{f}(a)= \bar{\tau}_u +
\tilde{f}(a')- \bar{\tau}_v$. If $n \geq 2$ and $\alpha \in F_n(u)$
then  $\alpha^{\tau_u} \in F_n(x)$ and we define $\tilde{f}(\alpha)=
\tilde{f}(\alpha^{\tau_u})^{-\bar{\tau}_u}$.

It is straightforward to check that these definitions give  a
morphism $\tilde{f}: F,x \to \widetilde{C},y$ of crossed complexes
lifting $f$, and the uniqueness of such a lift is also easy to
prove.
\end{proof}
We will use the above result in the following form.
\begin{cor}\label{cor:pullbackcrs}
Let $p\colon  \wt{C} \to C$ be a covering morphism of crossed
complexes, and let $F$ be a connected and simply connected crossed
complex. Then the following diagram, in which each $\epsilon$ is an
evaluation morphism,  is a pullback in the category of crossed
complexes:
$$ \xybiglabels \xymatrix{\Crs(F,\wt{C}) \times F_{\phantom{x}} \ar [r]^-\epsilon \ar[d]_{p_* \times 1}
& \wt{C} \ar[d]^p\\
\Crs(F,{C}) \times F \ar [r]_-\epsilon & C,  }$$ where the sets of
morphisms of crossed complexes have the discrete crossed complex
structure.
\end{cor}
\begin{proof}
This is simply a restatement of a special case of the existence and
uniqueness of liftings of morphisms established in the Proposition.
\end{proof}
\begin{rem}
Because the category $\Crs$ is equivalent to that of strict globular
$\omega$-groupoids, as shown in \cite{BH81:inf}, the methods of this
paper are also relevant to that category; see also
\cite{brownglobularhhgpd}. However we are not able to make use of
the globular case, nor even the $2$-groupoid case. \hfill $\Box$
\end{rem}

Let $C$ be a crossed complex. We write ${\Crs\mathsf{Cov}}/C$
\label{GSII5:CC49} for the full subcategory of the slice category
${\Crs}/C $ whose objects are the covering morphisms of $C$. The
following Theorem, which is proved in \cite{BRa99}, shows that the
classification of covering morphisms of crossed complexes, reduces
to that of covering morphisms of groupoids.

\begin{thm} \label{thm:II5-equivof categories}
If $C$ is a crossed complex, then the functor $\pi_1  \colon
{\Crs}\rightarrow {\Gpds}$ induces an equivalence of categories
 $$\pi_1'  \colon  {\Crs\mathsf{Cov}}/C \rightarrow {\Gpds\mathsf{Cov}}/(\pi_1C).$$
\end{thm}

An alternative descriptions of the category ${\Gpds\mathsf{Cov}}/G$
for a groupoid $G$ in terms of actions of $G$ on sets is well known
and of course gives the classical theory of covering maps of spaces,
see \cite[Chapter 10]{B2006}. Consequently, if the crossed complex
$C$ is connected, and $x \in C_0$, then connected covering morphisms
of $C$ are determined up to isomorphism by conjugacy classes of
subgroups of $\pi_1(C,x)$. In particular, a universal cover
$\tilde{C} \to C$ of a connected crossed complex is constructed up
to isomorphism from a base point $x \in C_0$ and the trivial
subgroup of $\pi_1(C,x)$.

The monoidal closed structure and many other major properties of
crossed complexes are obtained by working through another algebraic
category, that of {\itshape cubical} $\omega ${\itshape -groupoids}
with connections which we abbreviate here to $\omega ${\itshape
-groupoids}. The category of these, which we write $\omega $-Gpd, is
a natural home for these deeper properties. The equivalence with
crossed complexes proved in [BH81] is a foundation for this whole
project. Indeed the definition of tensor product for $\omega
$-groupoids is much easier to deal with than that for crossed
complexes, and we find it easier to give a dense subcategory for
$\omega $-groupoids than for crossed complexes.

\section{Cubical omega-groupoids with connection}
We recall from \cite{BH81:algcub} that a {\it cubical
$\omega$-groupoid with connection} is in the first instance a
cubical set $\{K_n \mid n \ge 0\}$, so that it has face maps
$\{\partial_i^\pm \colon  K_n \to K_{n-1}\mid i=1, \ldots, n; n \ge
1 \}$ and degeneracy maps $\{\eps_i\colon  K_{n} \to K_{n+1}\mid
i=1, \ldots, n; n \ge 0 \}$ satisfying the usual rules. Further
there are {\it connections} $\{\Gamma_i^\pm\colon K_{n} \to
K_{n+1}\mid i=1, \ldots, n; n \ge 1 \}$ which amount to an
additional family of `degeneracies' and which in the case of the
singular cubical complex of a space derive from the monoid
structures $\max,\min$ on the unit interval $[0,1]$. Finally there
are $n$ groupoid structures $\{\circ_i\mid i=1, \ldots, n\}$,
defined on $K_n$ with initial, final and identity maps
$\partial^-_i,
\partial^+_i, \eps_i$  maps respectively.

The laws satisfied by all these structures are given in several
places, such as \cite{ABS,Grandis-cubsite}, and we do not repeat
them here. Note that because we are dealing with groupoid operations
$\circ_i$ we can set $\Gamma_i=\Gamma_i^-$ so that $\Gamma^+_i=
-_i-_{i+1}\Gamma_i$. In this case the laws were first  given in
\cite{BH81:algcub}.

A major example of this structure  is constructed from a filtered
space $X_*$ as follows.  One first forms the cubical set with
connections $R X_*$ which in dimension $n$ is the set of filtered
maps $I^n_* \to X_*$ where $I^n_*$ is the standard $n$-cube with its
skeletal filtration. Then $\rho X_*$ is the quotient of $R X_*$ by
the relation of homotopy through filtered maps and relative to the
vertices of $I^n$. It is easy to see that $\rho X_*$ inherits the
structure of cubical set with connection, and it is proved in
\cite[Theorem A]{BH81:col} that the obvious compositions on $R X_*$
are also inherited by $\rho X_*$ to make it what is called the
fundamental $\omega$-groupoid $\rho X_*$ of the filtered space
$X_*$.

The main result of \cite{BH81:algcub} is that the category  \ogpd\
is equivalent to the category $\Crs$ of crossed complexes, and in
\cite[Theorem 5.1]{BH81:col} it is proved that this equivalence
takes $\rho X_*$ to $\Pi X_*$.

As said in the Introduction, the free $\omega$-groupoid on a
generator $c_n$ of dimension $n$ is written $\bI^n$. More generally,
the free $\omega$-groupoid on a cubical set $K$ is written $\rho'
K$: this is a purely algebraic definition. A major result  is that
$\rho' K$ is equivalent to $\rho |K|_*$ where $|K|_*$ is the
skeletal filtration of the geometric realisation of $K$ and $\rho$
is defined above; so we write both as $\rho K$. This equivalence is
proved in \cite[Proposition 9.5]{BH81:col} for the case $K=\I^n$,
and the general case follows by similar methods.

We shall also need the properties of thin elements in an
$\omega$--groupoid $G$. An element $t$ of $G_n$ is called {\it thin}
if it has a decomposition as a multiple composition of elements
$\epsilon_i x, \Gamma_j y$, or their repeated negatives in various
directions. Clearly a morphism of $\omega$--groupoids preserves thin
elements.

A family  $B$ of elements of $\bI^{n}$ is called an {\it $(n-1)$-box
in $\bI^n$} if they form all faces $\partial^\pm_i c_n$ but one of
$c_n$. An element $x$ is called a {\it filler} of the box if these
all-but-one faces $\partial^\pm_i x$ are exactly the elements of
$B$.

Then $B$ generates a sub-$\omega$-groupoid $\bar B$ of $\bI^n$. The
image family $\hat{b}({B})$ of this by a morphism of
$\omega$--groupoids $\hat{b}\colon  \bar{B} \to G$ is called an
$(n-1)$-box in $G$. Again we have the notion of a filler of a box in
$G$. A basic result on $\omega$--groupoids \cite[Proposition
7.2]{BH81:algcub} is:
\begin{prop}[Uniqueness of thin fillers] A box in an
$\omega$--groupoid has a unique thin filler.
\end{prop}
The thin elements in an $\omega$-groupoid satisfy Keith Dakin's
axioms, \cite{Dakin-thesis}:
\begin{enumerate}[D1)]
\item a degenerate element is thin;
\item every box has a unique thin filler;
\item if all faces but one of a thin element are thin, then so is
the remaining face.
\end{enumerate}
These axioms for a thin structure in fact give a structure
equivalent to that of an $\omega$--groupoid, as shown in
\cite{BH81:T}. That is, the connections {\it and the compositions}
are determined by the thin structure: we will use this fact in the
proof of Theorem \ref{thm:dense}. The following Lemma is also used
there.

\begin{lem}\label{lem:thinmorphism}
If $t \in G_n$ is a thin  element of an $\omega$--groupoid $G$ ,
then there is a thin element $b_t\in \bI^n$  such that
$\hat{t}(b_t)= \hat{t}(c_n)$.
\end{lem}
\begin{proof}
Let $\hat{t}: \bI^n \to G$ be the morphism such that $\hat{t}(c_n)=
t$.  We can find a box $B$ in $\bI^n$ and such that $t$ is a filler
of $\hat{t}\,|\colon \bar{B} \to G$. This box $B$ in $\bI^n$ also
has a unique thin filler $b_t$ in $\bI^n$. Since $\hat{t}$ is a
morphism of $\omega$-groupoids, it preserves thin elements and so
$\hat{t}(b_t)$ is thin and also a filler of the box $B$ in $G$. By
uniqueness of thin fillers $\hat{t}(b_t)=t =\hat{t}(c_n)$.
\end{proof}
\begin{rem}
Thin elements in higher categorical rather than groupoid situations
are also used in
\cite{street-orientedsimplexes,higgins-thin,steiner-thinelements,verity-comlicial}.
\hfill $\Box$
\end{rem}
\section{Fibrations and coverings of omega-groupoids}\label{sec:fibrandcov}We now transfer
to cubical $\omega$--groupoids the definition in Section
\ref{secII4:FibCrs} of fibration and covering morphism of crossed
complex.

\begin{thm}\label{fibration}
Let $p\colon G \to H$ be a morphism of \ogpd s. Then the
corresponding morphism of crossed complexes $\gamma(p)\colon
\gamma(G) \to \gamma(H) $ is a fibration (covering morphism) if and
only if $p\colon G \to H$ is a Kan fibration (covering map) of
cubical sets.
\end{thm}
\begin{proof}Let $J^n_{\eps,i}$ for $\eps=\pm, i=1,\ldots,n$,  be
the subcubical set  of the cubical set $I^n$ generated by all faces
of $I^n$ except $\partial^\eps_i$.

We consider the following diagrams:

$$\def\labelstyle{\textstyle}
\xymatrix{ \Pi J^n_{\eps,i} \ar [d] \ar [r] & \gamma G \ar [d]^{\gamma(p)} \\
\Pi I^n \ar @{.>}[ur] \ar [r] & \gamma H \\
\ar@{}[r] ^{(i)}&}\qquad\xymatrix{ \rho J^n_{\eps,i}
\ar [d] \ar [r] &  G \ar [d]^{p} \\
\rho I^n \ar @{.>}[ur] \ar [r] &  H \\ \ar@{}[r] ^{(ii)}&}
\qquad\xymatrix{ J^n_{\eps,i}
\ar [d] \ar [r] &  UG \ar [d]^{Up} \\
 I^n \ar @{.>}[ur] \ar [r] &  UH \\ \ar@{}[r] ^{(iii)}&}
$$
By a simple modification of the simplicial argument in \cite{BH91},
we find that the condition that diagrams of the first type  have the
completion shown by the dotted arrow is  necessary and sufficient
for $\gamma p$ to be a fibration of crossed complexes (with
uniqueness for a covering morphism). In the second diagram, $\rho(K)
$ is the free cubical $\omega$-groupoid on the cubical set $K$, and
the equivalence of the first and the second diagram is one of the
results of \cite[Section 9]{BH81:col}. Finally, the equivalence with
the third diagram, in which $U$ gives the underlying cubical set,
follows from freeness of $\rho$.
\end{proof}

\begin{cor} \label{cor:pullbackprescolimits}Let $p\colon  K \to L$ be a morphism of \ogpd s such that
the underlying map of cubical sets is a Kan fibration. Then the
pullback functor $$f^*\colon  \ogpdm/L \to \ogpdm/K$$ has a right
adjoint and so preserves colimits.
\end{cor}
\begin{proof} This is immediate from Theorem \ref{fibration} and
the main result of Howie \cite{How79}.
\end{proof}
\begin{cor}\label{cor:coveringoffree}
A covering crossed complex of a free crossed complex is also free.
\end{cor}
\begin{proof}
A free crossed complex is given by a sequence of pushouts,
analogously to the definition of CW-complexes,  see \cite{BH91,BHS}.
\end{proof}

\section{Dense subcategories}\label{sec:dense}
Our aim in this section is to explain and prove the theorem:

\begin{thm}\label{thm:dense}
The full subcategory $\mathscr{I}$ of \ogpd\ on the objects $ \bI^n$
is dense in \ogpd.
\end{thm}

We recall from \cite{Mac71} the definition of a dense subcategory.
First, in any category $\C$, a morphism $f\colon  C \to D$ induces a
natural transformation $f_*\colon  \C(-, C) \Rightarrow \C(-,D)$ of
functors $\C^{op} \to \Set$. Conversely, any such natural
transformation is induced by a (unique) morphism $C \to D$.

If $\sI$ is a subcategory of $\C$, then  each object $C$ of $\C$
gives a functor
$$\C^{|\sI}(-,C) \colon\sI^{op} \to \Set$$
and a morphism $f\colon C \to D$ of $\C$ induces a natural
transformation of functors $f_*\colon \C^{|\sI}(-,C) \Rightarrow
\C^{|\sI}(-,D)$. The subcategory $\sI$ is {\it dense} in $\C$ if
every such natural transformation arises from a morphism. More
precisely, there is a functor $\eta\colon  \C \to
\CAT(\sI^{op},\Set)$ defined in the above way, and $\sI$ is {\it
dense} in $\C$ if $\eta$ is full and faithful. \begin{example}
Consider the Yoneda embedding
$$\Upsilon\colon  \C \to \C^{op}\mbox{-}\Set= \CAT(\C^{op},\Set)$$
where $\C$ is a small category. Then each object $K \in
\C^{op}\mbox{-}\Set$ is a colimit of objects in the image of
$\Upsilon$ and this is conveniently expressed in terms of coends as
that  the natural morphism
$$\int^c\left(\C^{op}\mbox{-}\Set(\Upsilon c, K) \times \Upsilon c\right) \quad\to \quad K$$ is
an isomorphism. Thus the Yoneda image of $\C$ is dense in
$\C^{op}\mbox{-}\Set$. For more on the relation between density and
the Yoneda Lemma, see \cite{pratt-yoneda}.  \hfill $\Box$
\end{example}
\begin{example}
Let $\bZ$ be the cyclic group of integers. Then  $\{\bZ\}$ is a
generating  set for the category $\Ab$ of abelian groups, but the
full subcategory of $\Ab$ on this set is not dense in $\Ab$. In
order for a natural transformation to specify not just a function
$f\colon  A \to B$ but  a morphism in $\Ab$, we have  to enlarge
this to a full subcategory including $\bZ \oplus \bZ $.   \hfill
$\Box$
\end{example}

\noindent {\bf Proof of Theorem \ref{thm:dense}} We will use the
main result of \cite{BH81:T}, that the compositions in a cubical
$\omega$-groupoid are determined by its thin elements.

Let $G,H$ be $\omega$--groupoids and let $\mathsf{f}\colon
\ogpdm_\sI(-,G) \to \ogpdm_\sI(-,H)$ be a natural transformation. We
define $f\colon G \to H$ as follows.

Let $x \in G_n$. Then $x$ defines $\hat{x}\colon  \bI^n \to G$. We
set $f(x)= \mathsf{f}(\hat{x})(c^n) \in H_n$. We have to prove $f$
preserves all the structure.

For example, we prove that $f(\partial ^\pm_i x)= \partial
^\pm_if(x)$. Let $\bar{\partial}^\pm_i\colon  \bI^{n-1} \to \bI^n$
be given by having value $\partial^\pm_i c^n$ on $c^{n-1}$. The
natural transformation condition implies that
$\mathsf{f}(\bar{\partial}^\pm_i)^*=
(\bar{\partial}^\pm_i)^*\mathsf{f}$. On evaluating this on $\hat{x}$
we obtain $f(\partial ^\pm_i x)=\partial ^\pm_if(x)$ as required. In
a similar way, we prove that $f$ preserves the operations
$\epsilon_i, \Gamma_i$.

Now suppose that $t \in G_n$ is thin in $G$. We prove that $f(t)$ is
thin in $H$. By Lemma \ref{lem:thinmorphism}, there is a thin
element $b_t\in \bI^n$ such that $\hat{t}(b_t)=t$. Let
$\bar{b}\colon \bI^n \to \bI^n$ be the unique morphism such that
$\bar{b}(c^n)= b_t$. Then the natural transformation condition
implies $f(t)=\mathsf{f}(\hat{t})(c^n)=\mathsf{f}(\hat{t})(b_t)$.
Since $b_t$ is thin, it follows that $f(t)$is thin. Thus $f$
preserves the thin structure.

The main result of \cite{BH81:T} now implies that the operations
$\circ_i$ are preserved by $f$. \hfill $\Box$

We can also conveniently represent each $\omega$--groupoid as a
coend.
\begin{cor}\label{cor:omegagroupoidascolimit}
The subcategory  $ \sI $ of $\ogpdm$ is dense and for each object
$G$ of \ogpd\ the natural morphism
\begin{equation*}
  \int ^n \ogpdm(\bI^n,G) \times \bI^n \to G
\end{equation*}
is an isomorphism.
\end{cor}
\begin{proof}
  This is a standard consequence of the property of $\sI$ being dense.
\end{proof}\begin{cor}\label{cor:densecrossedcompl}
The full subcategory of $\Crs$ generated by the objects $\Pi I^n_*$
is dense in $\Crs$.
\end{cor}
\begin{proof}
This follows from the fact that the equivalence $\gamma\colon \ogpdm
\to \Crs$ takes $\bI^n$ to $\Pi I^n_*$, \cite[Theorem
5.1]{BH81:col}.
\end{proof}
\begin{rem}
The paper \cite{BH81:inf} gives an equivalence between the category
$\Crs$ of crossed complexes and the category there called
$\infty$-groupoids and now commonly called globular
$\omega$-groupoids. Thus the above Corollary yields also a dense
subcategory, based on models of cubes, in  the latter category.
\hfill $\Box$
\end{rem}
\begin{rem}
It is easy to find a generating set of objects for the category
$\Crs$, namely the free crossed complexes on single elements, given
in fact by $\Pi E^n_*$, where $E^n_*$ is the usual cell
decomposition of the unit ball, with one cell for $n=0$ and
otherwise three cells. It is not so obvious how to construct
directly from this generating set a dense subcategory closed under
tensor products. \hfill $\Box$
\end{rem}


\section{The tensor product of covering morphisms}\label{sec:tensor}

Our aim is to prove the  following:
\begin{thm}\label{thm:IIItenscov}
The tensor product of two covering morphisms of crossed complexes is
a covering morphism.
\end{thm}
\begin{rem}
The reason why we have to give  an indirect proof of this result is
that the definition of  covering morphism involves {\it elements} of
crossed complexes; but it is difficult to specify exactly the
elements of a tensor product whose definition is perforce by
generators and relations. \hfill $\Box$
\end{rem}
It is sufficient to assume that all the crossed complexes involved
are connected. We will also work in the category of $\omega$--groupoids, and prove the following:

\begin{thm} \label{thm:covomeg}
Let $G,H$ be connected $\omega$--groupoids with base points $x,y$
respectively, and let $p: \tilde{G} \to G$ be the covering morphism
determined by the subgroup $M$ of $\pi_1(G,x)$. Let $\phi: C \to G
\otimes H$ be the covering morphism determined by the subgroup $M
\times \pi_1(H,y)$ of $$ \pi_1(G \otimes H,(x,y))\cong \pi_1(G,x)
\times \pi_1(H,y).$$ Then there is an isomorphism $\psi: C \to
\tilde{G} \otimes H$ such that $(p\otimes 1_H) \psi= \phi$, and,
consequently,
$$p\otimes 1_H: \tilde{G} \otimes H \to G \otimes H$$ is a covering
morphism.
\end{thm}
\noindent {\bf Proof} Here we were inspired by the formulae
 of Brian Day \cite{Day-thesis}.

First we know from \cite{BH87} that the tensor product of $\omega
$-Gpds satisfies $\mathbb{I}^{m}\otimes \mathbb{I}^{n}\cong
\mathbb{I}^{m+n}$, showing that $\mathcal{I}$ is a full monoidal
subcategory of $\omega $-Gpds. Since also from \cite{BH87} the
tensor preserves colimits in each variable, it follows from
Corollary \ref{cor:omegagroupoidascolimit}  that the tensor product
$G\otimes H$ of $\omega $-groupoids $G$ and $H$ satisfies
\begin{equation}
G \otimes H \cong \int^{m,n} \ogpdm(\bI^m,G) \times \ogpdm(\bI^n,H)
\times (\bI^m \otimes \bI^n) .
\end{equation}

Let $p \colon \tilde{G} \to G$ be the covering morphism determined
by the subgroup $M$ and let $\phi \colon C \to G \otimes H$ be the
covering morphism determined by  the subgroup $M \times \pi_1(H,y)$
of
$$\pi_1(G,x)\times \pi_1(H,y) \cong \pi_1(G \otimes H, (x,y)).$$  By Corollary \ref{cor:pullbackprescolimits},
pullback $\phi^*$ by $\phi$ preserves colimits. Hence
\begin{align*}
C &\cong \phi^*\left(\int^{m,n} \ogpdm(\bI^m,G) \times
\ogpdm(\bI^n,H)\times
 (\bI^m \otimes \bI^n)\right)\\
& \cong\int^{m,n}\phi^*(\ogpdm(\bI^m,G) \times \ogpdm(\bI^n,H))
\times
(\bI^m \otimes \bI^n)\\
\intertext{and so  because of the construction of $C$ by the
specified subgroup:} & \cong\int^{m,n}\ogpdm(\bI^m,\tilde{G})
\times\ogpdm(\bI^m,{H})\times(\bI^m \otimes \bI^n)\\
&\cong\tilde{G}\otimes {H}. \tag*{$\Box$}
\end{align*}
\begin{cor}
The tensor product of covering morphisms of $\omega$-groupoids is
again a covering morphism.
\end{cor}
\begin{proof}
Because tensor product commutes with disjoint union, it is
sufficient to restrict to the connected case. Since the composition
of covering morphisms is again a covering morphism, it is sufficient
to restrict to the case of $p \otimes 1_H$, and that is proved  in
Theorem \ref{thm:covomeg}.
\end{proof}
The proof of Theorem \ref{thm:IIItenscov} follows immediately.

\begin{cor}\label{cor:freecrossedres}
If $F,F'$ are free and aspherical crossed complexes, then so also is
$F \otimes F'$.
\end{cor}
\begin{proof}
It is sufficient to assume $F,F'$ are connected. Since $F,F'$ are
aspherical, their universal covers $\tilde{F}, \tilde{F'}$ are
acyclic. Since they are also free, they are contractible, by a
Whitehead type theorem,  \cite[Theorem 3.2]{BG89}. But the tensor
product of free crossed complexes is free, by \cite[Cor. 5.2]{BH91}.
Therefore $\tilde{F}\otimes \tilde{F'}$ is contractible, and hence
acyclic. Therefore $F \otimes F'$ is aspherical.
\end{proof}

\noindent {\it Acknowledgement} We thank a referee for helpful
comments.

\end{document}